\documentclass[11pt]{article}
\usepackage{amssymb, amsmath, amsthm, amsfonts,xcolor, enumerate}
\usepackage{tikz}
\usetikzlibrary{patterns}
\usetikzlibrary{shapes}
\usepackage{fullpage}
\usepackage{hyperref}
\newtheorem{theorem}{Theorem}[section]
\newtheorem{ques}[theorem]{Question}

\newtheorem{lemma}[theorem]{Lemma}
\newtheorem{cor}[theorem]{Corollary}

\newtheorem{conjecture}[theorem]{Conjecture}
\newtheorem{prop}[theorem]{Proposition}

\newtheorem{fact}[theorem]{Fact}
\newtheorem{claim}[theorem]{Claim}
\theoremstyle{definition}

\theoremstyle{definition}

\theoremstyle{definition}

\theoremstyle{definition}

\theoremstyle{definition}

\theoremstyle{definition}

\theoremstyle{definition}

\theoremstyle{definition}

\newcommand{\ep}{\varepsilon}
\newcommand{\eps}{\varepsilon}

\newcommand{\de}{\delta}

\newcommand{\ga}{\gamma}
\newcommand{\Ga}{\Gamma}
\newcommand{\De}{\Delta}

\newcommand{\cM}{\mathcal{M}}

\newcommand{\cF}{\mathcal{F}}
\newcommand{\cI}{\mathcal{I}}

\newcommand{\mis}{\mathrm{MIS}}
\newcommand{\fm}{f_{\max}}
\newcommand{\msf}{\mathrm{MSF}}

\def\COMMENT#1{}
\let\COMMENT=\footnote
\title{Sharp bound on the number of maximal sum-free subsets of integers}

\author{
 J\'ozsef Balogh
 \thanks{Department of Mathematical Sciences,
 University of Illinois at Urbana-Champaign, Urbana, Illinois 61801, USA {\tt
jobal@math.uiuc.edu}. Research is  partially supported by NSF Grant DMS-1500121 and an Arnold O. Beckman Research Award (UIUC Campus Research Board 15006).
}
 \quad
 Hong Liu
 \thanks{Mathematics Institute,
 University of Warwick, United Kingdom, {\tt
h.liu.9@warwick.ac.uk}. Research is supported by the Leverhulme Trust Early Career Fellowship~ECF-2016-523.}
 \quad
 Maryam Sharifzadeh
 \thanks{Mathematics Institute,
 	University of Warwick, United Kingdom,  {\tt
m.sharifzadeh@warwick.ac.uk}. Research is supported by the European Unions Horizon 2020 research and innovation programme under the Marie Curie Individual Fellowship No.~752426.}
\quad
Andrew Treglown
\thanks{School of Mathematics, University of Birmingham, United Kingdom, {\tt 
a.c.treglown@bham.ac.uk}. Research is supported by EPSRC grant EP/M016641/1.}
}

\begin{document}
\maketitle

\begin{abstract}
Cameron and Erd\H{o}s~\cite{CE} asked whether  the number of \emph{maximal} sum-free sets in $\{1, \dots , n\}$ is much smaller than the  number of sum-free sets. In the same paper they gave a lower bound of $2^{\lfloor n/4 \rfloor }$ for the number of maximal sum-free sets. Here, we prove the following: For each $1\leq i \leq 4$, there is a constant $C_i$ such that, given any $n\equiv i \mod 4$,
$\{1, \dots , n\}$ contains  $(C_i+o(1)) 2^{n/4}$ maximal sum-free sets.
 Our proof makes use of  container and removal lemmas of Green~\cite{G-CE, G-R},  a structural result of Deshouillers, Freiman, S\'os and Temkin~\cite{DFST} and a recent bound on the number of subsets of integers with small sumset by Green and Morris~\cite{GM}. We also discuss related results and open problems on the number of maximal sum-free subsets of abelian groups.
\end{abstract}


\section{Introduction}

A triple $x,y,z$ is a \emph{Schur triple} if $x+y=z$ (note $x,y$ and $z$ may not necessarily be distinct).
A set $S$ is \emph{sum-free} if $S$ does not contain a Schur triple. Let $[n]:=\{1,\dots,n\}$. We say that $S\subseteq [n]$ is a \emph{maximal sum-free subset of $[n]$} if it is sum-free and it is not properly contained in another sum-free subset of $[n]$. Let $f(n)$ denote the number of sum-free subsets of $[n]$ and $f_{\max}(n)$ denote the number of maximal sum-free subsets of $[n]$. The study of sum-free sets of integers has a rich history. Clearly, any set of odd integers and any subset of $\{\lfloor n/2 \rfloor +1, \dots , n\}$ is a sum-free set, hence $f(n)\ge 2^{n/2}$. Cameron and Erd\H{o}s~\cite{cam1} conjectured that $f(n)=O(2^{n/2})$. In fact, they conjectured the stronger statement that $f(n)/2^{n/2}$ tends to two different constants depending on the parity of $n$.\footnote{The constants for odd (resp. even) $n$ is approximately $6.8$ (resp. 6.0).} This conjecture was proven independently by Green~\cite{G-CE} and Sapozhenko~\cite{sap}. Indeed, they showed that there are constants $C_1$ and $C_2$ such that $f(n)=(C_i+o(1))2^{n/2}$ 
for all $n \equiv i \mod 2$.

In a second paper, Cameron and Erd\H{o}s~\cite{CE} observed that $f_{\max} (n) \geq 2^{\lfloor n/4 \rfloor }$. Noting that all the sum-free subsets of $[n]$ described above lie in just two maximal sum-free sets, they asked whether $f_{\max} (n) = o(f(n))$ or even $f_{\max} (n) \leq f(n)/ 2^{\ep n}$ for some constant $\ep >0$. \L{u}czak and Schoen~\cite{ls} answered this question in the affirmative, showing that $f_{\max} (n) \leq 2^{n/2-2^{-28}n}$ for sufficiently large $n$. Later, Wolfovitz~\cite{wolf} proved that $f_{\max} (n) \leq 2^{3n/8+o(n)}$. More recently, the authors~\cite{BLST} proved that the lower bound is essentially tight, proving that $f_{\max}(n)=2^{(1/4+o(1))n}$. In this paper we give the following exact solution to the problem. 


\begin{theorem}\label{thm-main}
For each $1\leq i \leq 4$, there is a constant $C_i$ such that, given any $n\equiv i \mod 4$,
$[n]$ contains  $(C_i+o(1)) 2^{n/4}$ maximal sum-free sets.
\end{theorem}

We remark that for the constants $C_i$ can also be computed up to any additive error (say $\eps$) in constant time (i.e.~depending only on $\eps$). We refer the reader to Section~\ref{explain} (and the remarks after Lemma~\ref{lem-intmain}) for more details.
The proof of Theorem~\ref{thm-main} is given in Section~\ref{sec-main}, with the main work arising in Section~\ref{subsec-interval}.
The proof draws on a number of ideas from~\cite{BLST}. In particular, as in~\cite{BLST} we 
 make use of `container' and `removal' lemmas of Green~\cite{G-CE, G-R} as well as a  result of Deshouillers, Freiman, S\'os and Temkin~\cite{DFST} on the structure of sum-free sets. 
Our work also has parallels with recent developments on maximal triangle-free graphs~\cite{BLPS,sar} (see the introduction in~\cite{BLPS} for a discussion on this).

Despite these connections, the  details of these proofs are actually significantly different to the proof of Theorem~\ref{thm-main}. In particular, as described in Section~\ref{21}, the container method is naturally set up to yield an error term in the exponent when computing $f_{\max} (n)$. Thus,
in order to avoid over-counting the number of maximal sum-free subsets of $[n]$,
our present proof  develops a number of new ideas, thereby making the argument substantially more involved.
We use a bound on the number of subsets of integers with small sumset by Green and Morris~\cite{GM} as well as several new bounds on the number of maximal independent sets in various graphs.
Further, the proof provides  information about the typical structure of the maximal sum-free subsets of $[n]$. Indeed, we show that almost all of the maximal sum-free subsets of $[n]$ look like one of two particular extremal constructions (see Section~\ref{23} for more details).

Our main result is an example of an \emph{enumeration} problem. This area has a long history. In particular, in the context of graph theory, the study was initiated by Erd\H{o}s, Kleitman and Rothschild~\cite{ekr} who (up to an error term in the exponent) determined the number of $K_r$-free graphs on $n$ vertices. Since then, a number of tools have been developed for attacking such problems. However, progress on enumeration problems for sum-free sets has been slower. Indeed, as mentioned above, it took nearly 15 years for the conjecture of Cameron and Erd\H{o}s on the number of sum-free subsets of $[n]$ to be fully resolved. We believe that our methods  are likely to provide insight for attacking related problems. For example, in Section~\ref{group} we state several open problems on the number of maximal sum-free subsets of abelian groups.

In Section~\ref{overview} we give an overview of the proof and highlight the new ideas that we develop. We state  some useful results in Section~\ref{tools} and prove Theorem~\ref{thm-main} in Section~\ref{sec-main}.

\section{Background and an overview of the proof of Theorem~\ref{thm-main}}\label{overview}
\subsection{Independence and container theorems}\label{21}
An exciting recent development has been the emergence of `independence' providing a framework to study a plethora of problems arising in combinatorics, geometry, number theory and probability as well as at the interfaces of such areas. 
To be more precise, let $V$ be a set and $\mathcal E$ a collection of subsets of $V$. We say that a subset $I$ of $V$ is an \emph{independent set} if $I$ does not contain any element of $\mathcal E$ as a subset. 
For example, if $V:=[n]$ and  $\mathcal E$ is the collection of all Schur triples in $[n]$ then an independent set $I$ is simply a sum-free set. 
It is often helpful to think of $(V,\mathcal E)$ as a hypergraph with vertex set $V$ and edge set $\mathcal E$; thus an independent set $I$ corresponds to an independent set in the hypergraph.

So-called `container results' have  emerged as a powerful tool for attacking many problems that concern counting independent sets. Roughly speaking, container results state that the independent sets of a given hypergraph $H$ lie only in a `small' number of subsets of the vertex set of $H$ (referred to as \emph{containers}), where each of these containers is an `almost independent set'.  Balogh, Morris and Samotij~\cite{BMS} and independently Saxton and Thomason~\cite{ST},  proved  general container theorems  for hypergraphs whose edge distribution satisfies  certain boundedness conditions. 

In the proof of Theorem~\ref{thm-main} we will apply the following container theorem of Green~\cite{G-CE}.
\begin{lemma}[Proposition 6 in~\cite{G-CE}]\label{lem-container}
There exists a family $\cF$ of subsets of $[n]$ with the following properties.

(i) Every member of $\cF$ has at most $o(n^2)$ Schur triples.

(ii) If $S\subseteq [n]$ is sum-free, then $S$ is contained in some member of $\cF$.

(iii) $|\cF|=2^{o(n)}$.

(iv) Every member of $\cF$ has size at most $(1/2+o(1))n$.
\end{lemma}
We refer to the sets in $\cF$ as \emph{containers}. 

In~\cite{BLST} we used Lemma~\ref{lem-container} to prove that $f_{\max}(n)=2^{(1+o(1))n/4}.$ Indeed, we showed that every  $F \in \mathcal F$ contains at most
$2^{(1+o(1))n/4}$ maximal sum-free subsets of $[n]$ which by (ii) and (iii) yields the desired result. 
To obtain an exact bound on $f_{\max}(n)$ it is not sufficient to give a tight general bound on the number of maximal sum-free subsets of $[n]$ that lie in a container $F \in \mathcal F$.
Indeed, such an $F \in \mathcal F$ could contain $O(2^{n/4})$ maximal sum-free subsets of $[n]$, and thus together with (iii) this still gives an error term in the exponent.
In general, since containers may overlap,  applications of container results may lead to `over-counting'.

We therefore need to count the number of maximal sum-free subsets of $[n]$ in a more refined way. To explain our method, we first need to describe the constructions which imply that 
 $f_{\max} (n) \geq 2^{\lfloor n/4 \rfloor }$.

\subsection{Lower bound constructions}\label{extremal}
The following construction of Cameron and Erd\H{o}s~\cite{CE} implies that $f_{\max} (n) \geq 2^{\lfloor n/4 \rfloor }$. Let $n \in \mathbb N$ and 
 let $m=n$ or $m=n-1$, whichever is even. Let $S$ consist of $m$ together with precisely
one number from each pair  $\{x,m-x\}$ for odd $x <m/2$. Then $S$ is sum-free. Moreover, although $S$ may not be maximal, no further odd numbers less than $m$ can be added, so distinct $S$ lie in distinct maximal sum-free subsets of $[n]$.

The following construction from~\cite{BLST} also yields the same lower bound on $f_{\max} (n)$.
Suppose that $4|n$ and set $I_1:=\{n/2+1, \ldots, 3n/4\}$ and $I_2:=\{3n/4+1, \ldots,n\}$. First choose the element $n/4$ and a set $S'\subseteq I_2$. Then for every $x\in I_2\setminus S'$, choose $x-n/4\in I_1$. The resulting set $S$ is sum-free but may not be maximal. However, no further element in $I_2$ can be added, thus distinct $S$ lie in distinct maximal sum-free sets in $[n]$. There are $2^{|I_2|}=2^{n/4}$ ways to choose $S$.

\subsection{Counting maximal sum-free sets}\label{23}
The following result provides structural information about the containers $F \in \mathcal F$.  Lemma~\ref{lem-structure} is implicitly stated in~\cite{BLST} and was essentially proven in~\cite{G-CE}. It
is an immediate consequence of a result of Deshouillers, Freiman, S\'os and Temkin~\cite{DFST} on the structure of sum-free sets and a removal lemma of Green~\cite{G-R}. Here $O$ denotes the set of odd numbers in $[n]$.

\begin{lemma}\label{lem-structure}
If $F\subseteq [n]$ has $o(n^2)$ Schur triples  then either

(a) $|F|\leq 0.47 n$;

or one of the following holds for some $-o(1) \leq \ga=\ga(n)\leq 0.03$:

(b) $|F|=\left(\frac{1}{2}-\ga\right)n$ and $F=A \cup B$ where $|A|=o(n)$ and $B \subseteq [(1/2-\gamma)n, n]$ is sum-free;

(c) $|F|=\left(\frac{1}{2}-\ga\right)n$ and $F=A \cup B$ where $|A|=o(n)$ and $B \subseteq O$.
\end{lemma}

The crucial idea in the proof of Theorem~\ref{thm-main} is that we show `most' of the maximal sum-free subsets of $[n]$ `look like' the examples given in Section~\ref{extremal}: We first show that containers of type (a) house only a small (at most $2^{0.249n}$) number of maximal sum-free subsets of $[n]$ (see Lemma~\ref{lem-non-ext}).
For type (b) containers we split the argument into two parts. More precisely, we count the  number of maximal sum-free subsets $S$ of $[n]$ with the property that (i)  the smallest element of $S$ is $n/4\pm o(n)$ 
and (ii) the second smallest element of $S$ is at least $n/2 -o(n)$. (For this we use a direct argument rather than counting such sets within the containers.)
We then show that the  number of maximal sum-free subsets of $[n]$ that lie in type (b) containers but that fail to satisfy  one of (i) and (ii) is small ($o(2^{n/4})$).
We use a similar idea for type (c) containers. Indeed, we show directly that the number of maximal sum-free subsets of $[n]$ that contain \emph{at most} one even number is $O(2^{n/4})$. We then show that the  number of maximal sum-free subsets of $[n]$ that lie in type (c) containers and which contain two or more even numbers is small ($o(2^{n/4})$).

In each of our cases, we give an upper bound on the number of maximal sum-free sets in a container by
counting the number of maximal independent sets in various auxiliary graphs. (Similar techniques were used in~\cite{wolf,BLST}, and in the graph setting in~\cite{sar}.)
In Section~\ref{aux} we collect together a number of results that are useful for this.


\section{Notation and preliminaries}\label{tools}

\subsection{Notation}
For a set $F\subseteq [n]$, denote by $\msf(F)$ the set of all maximal sum-free subsets of $[n]$ that are contained in $F$ and let $f_{\max}(F):=|\msf(F)|$. Also, denote by $\min(F)$ and $\max(F)$ the minimum and the maximum element of $F$ respectively. Let $\min _2 (F)$ denote the second smallest element of $F$. Denote by $E$ the set of all even and by $O$ the set of all odd numbers in $[n]$. 
Given  sets $A,B$, we let $A+B:= \{ a+b : \, a\in A, \, b \in B\}$.
 We say a real valued function $f(n)$ is exponentially smaller than another real valued function $g(n)$ if there exists a constant $\ep >0$ such that $f(n)\le g(n)/2^{\ep n}$ for $n$ sufficiently large.
We use $\log$ to denote the logarithm function of base $2$. 


Throughout, all graphs considered are simple unless stated otherwise. We say that $G$ is a \emph{graph possibly with loops} if $G$ can be obtained from a simple graph by adding at most one loop at each vertex.  We write $e(G)$ for the number of edges in $G$.
Given a vertex $x$ in $G$, we write $\deg _G (x)$ for the \emph{degree} of  $x$ in $G$. Note that a loop at $x$ contributes two to the degree of $x$. We write $\delta (G)$ for the \emph{minimum degree} and $\Delta (G)$ for the \emph{maximum degree} of $G$.
 Denote by $G[T]$ the induced subgraph of $G$ on the vertex set $T$ and $G\setminus T$ the induced subgraph of $G$ on the vertex set $V(G)\setminus T$. 
Given $x \in V(G)$, we write $N_G (x)$ for the \emph{neighourhood of $x$ in $G$}. Given $S\subseteq V(G)$, we write $N_G (S)$ for the set of vertices $y \in V(G)$ such that $xy \in E(G)$ for some $x \in S$. 
 
 We write $C_m$ for the cycle, and $P_m$ for the path on $m$ vertices. Given graphs $G$ and $H$ we write $G\square H$ for the \emph{cartesian product graph}. So $G\square H$ has vertex set
 $V(G) \times V(H)$ and $(x,y)$ and $(x',y')$ are adjacent in $G\square H$ if (i) $x=x'$ and $y$ and $y'$ are adjacent in $H$ or (ii) $y=y'$ and $x$ and $x'$ are adjacent in $G$. 

Throughout the paper we omit floors and ceilings where the argument is unaffected. We write $0<\alpha \ll \beta \ll \gamma$ to mean that we can choose the constants
$\alpha, \beta, \gamma$ from right to left. More
precisely, there are increasing functions $f$ and $g$ such that, given
$\gamma$, whenever we choose some $\beta \leq f(\gamma)$ and $\alpha \leq g(\beta)$, all
calculations needed in our proof are valid. 
Hierarchies of other lengths are defined in the obvious way.


\subsection{The number of sets with small sumset}

We  need the following lemma of Green and Morris~\cite{GM}, which bounds the number of sets with small sumset.
\begin{lemma}\label{lem-G-M}
Fix $\de>0$ and $R>0$. Then the following hold for all integers $s\ge s_0(\de, R)$. For any $D\in\mathbb{N}$ there are at most 
$$2^{\de s}{\frac{1}{2}Rs\choose s}D^{\lfloor R+\de\rfloor}$$
sets $S\subseteq [D]$ with $|S|=s$ and $|S+S|\le R|S|$.
\end{lemma}

\subsection{Maximal independent sets in graphs}\label{aux}
In this section we collect together  results on the number of maximal independent sets in a graph. Let ${\rm{MIS}}(G)$ denote the number of maximal independent sets in a graph $G$.

Moon and Moser~\cite{MM} showed that for any simple graph $G$, ${\rm{\rm{MIS}}}(G)\le 3^{|G|/3}$. 
When a graph is triangle-free, this bound can be improved significantly: A result of Hujter and Tuza~\cite{HT} states that for any triangle-free graph $G$,
 \begin{align}\label{htnew}
 {\rm{\rm{MIS}}}(G)\le 2^{|G|/2}.
 \end{align}
The next result implies that the bound given in (\ref{htnew}) can be further lowered if $G$ is additionally not too sparse. 

\begin{lemma}\label{dense}
Let $n,D \in \mathbb N$ and $k \in \mathbb R$. 
Suppose that $G$ is a triangle-free graph  on $n$ vertices with $\Delta (G) \leq D$ and $e(G) \geq n/2+k$. Then 
$${\rm MIS} (G) \leq 2^{n/2-k/(100D^2)}.$$
\end{lemma}

The following result for `almost triangle-free' graphs follows from Lemma~\ref{dense}.
\begin{cor}\label{use}
Let $n,D \in \mathbb N$ and $k \in \mathbb R$. 
Suppose that $G$ is a graph and $T$ is a set such that $G':=G\setminus T$ is triangle-free.
Suppose that $\Delta (G) \leq D$, $|G'|=n$  and $e (G') \geq n/2+k$. Then 
$${\rm MIS} (G) \leq 2^{n/2-k/(100D^2)+101|T|/100}.$$
\end{cor}
We defer the proofs of Lemma~\ref{dense} and Corollary~\ref{use} to the appendix.

The following result gives an improvement on the Moon--Moser bound for graphs that are not too sparse, almost regular and of large minimum degree. (The result is proven as equation (3) in~\cite{BLST}.)
\begin{lemma}[\cite{BLST}] \label{lem-ik}
Let $k\ge 1$ and let $G$ be a graph on $n$ vertices possibly with loops. 
Suppose that $\De(G)\le k\de(G)$ and set $b:= \sqrt{\delta (G)}$.  Then
\begin{align*}
{\rm{MIS}}(G)\le \sum_{0 \le i\le n/b}{n\choose i}3^{{\left(\frac{k}{k+1}\right)\frac{n}{3}} + \frac{2n}{3b}}  .
\end{align*}
\end{lemma}

\begin{fact}\label{newfact}
Suppose that $G'$ is a (simple) graph. If $G$ is a graph obtained from $G'$ by adding loops at some vertices $x \in V(G')$ then
$${\rm MIS}(G) \leq {\rm MIS}(G').$$
\end{fact}

The following lemma from~\cite{BLPS} gives an improvement on~\eqref{htnew} when $G$ additionally contains many vertex disjoint $P_3$s. Its proof is similar to that of Lemma~\ref{dense}.

\begin{lemma}[\cite{BLPS}]\label{lem-mis-p3}
Let $G$ be an $n$-vertex triangle-free graph, possibly with loops. If $G$ contains $k$ vertex-disjoint $P_3$s, then 
$${\rm{MIS}}(G)\le 2^{\frac{n}{2}-\frac{k}{25}}.$$
\end{lemma}


\section{Proof of Theorem~\ref{thm-main}}\label{sec-main}
Let $1\leq i \leq 4$ and $0<\eta <1$.
To prove Theorem~\ref{thm-main}, we must show that there is a constant $C_i$ (dependent only on $i$) such that if $n$ is sufficiently large and $n \equiv i \mod 4$ then
\begin{align}\label{ultimateaim}
(C_i - \eta )2^{n/4} \leq f_{\max} (n) \leq (C_i + \eta )2^{n/4}.
\end{align}

Given $\eta >0$ and  sufficiently large $n$ with $n \equiv i \mod 4$, define constants $\alpha, \delta, \ep>0$ so that
\begin{align}\label{hier}
0<1/n \ll \alpha \ll  \delta \ll \ep \ll \eta <1.
\end{align}

Let $\mathcal F$ be the family of containers obtained from Lemma~\ref{lem-container}. Since $n$ is sufficiently large, Lemma~\ref{lem-structure} implies that $|\mathcal F|\leq 2^{\alpha n}$ and
for every $F \in \cF$ either

(a) $|F|\leq 0.47 n$;

or one of the following holds for some $-\alpha \leq \ga=\ga(n)\leq 0.03$:

(b) $|F|=\left(\frac{1}{2}-\ga\right)n$ and $F=A \cup B$ where $|A|\leq \alpha n$ and $B \subseteq [(1/2-\gamma)n, n]$ is sum-free; 

(c) $|F|=\left(\frac{1}{2}-\ga\right)n$ and $F=A \cup B$ where $|A|\leq \alpha n$ and $B \subseteq O$.

\noindent Throughout the rest of the paper we refer to such containers as type~(a), type~(b) and type~(c), respectively.

For any subsets $B, S\subseteq [n]$, let $L_S[B]$ be the \emph{link graph of $S$ on $B$} defined as follows. The vertex set of $L_S[B]$ is $B$. The edge set of $L_S[B]$ consists of the following two types of edges:

\noindent (i) Two vertices $x$ and $y$ are adjacent if there exists an element $z\in S$ such that $\{x,y,z\}$ forms a Schur triple; 

\noindent (ii) There is a  loop at a vertex $x$ if $\{x,x, z\}$ forms a Schur triple for some $z \in S$ or if  $\{x, z,z'\}$ 
forms a Schur triple for some $z, z'\in S$.

\smallskip

The following simple lemma from~\cite{BLST} will be applied in many cases throughout the proof.

\begin{lemma}[\cite{BLST}]\label{lem-mis}
Suppose that $B$ and $S $ are both sum-free subsets of $[n]$. If $I\subseteq B$ is  such that $S\cup I$ is a maximal sum-free subset of $[n]$, then $I$ is a maximal independent set in $G:=L_S[B]$.
\end{lemma}
The  next lemma will allow us to apply (\ref{htnew}) to certain link graphs.
\begin{lemma}\label{free}
Suppose that $B,S \subseteq [n] $ such that $S$ is sum-free and $\max (S) <\min(B)$. Then $G:=L_S[B]$ is triangle-free.
\end{lemma}
\proof
Suppose to the contrary that $z>y>x>\max(S)$ form a triangle in $G$. Then there exists $a,b,c\in S$ such that $z-y=a, y-x=b$ and $z-x=c$, which implies $a+b=c$ with $a,b,c\in S$. This is a contradiction to $S$ being sum-free.
\endproof
In the proof  we will use the simple fact that  if $S\subseteq T \subseteq [n]$ then 
\begin{align}\label{nondec}
f_{\max}(S)\le f_{\max}(T).
\end{align} 

The following lemma is a slightly stronger form of Lemma~$3.2$ from~\cite{BLST}, which deals with containers of `small' size. The proof is exactly the same as in~\cite{BLST}.

\begin{lemma}\label{lem-non-ext}
If $F\in\cF$ has size at most $0.47n$, then $f_{\max}(F)\leq 2^{0.249n}$.
\end{lemma}

Thus, to show that (\ref{ultimateaim}) holds it suffices to show that there is a constant $C_i$ such that 
in total, type (b) and (c) containers house $(C_i\pm \eta /2)2^{n/4}$ maximal sum-free subsets of $[n]$. In Section~\ref{subsec-interval} we deal with containers of type~(b) and in Section~\ref{subsec-odd} we deal with containers of type~(c).


\subsection{Type (b) containers}\label{subsec-interval}

The following lemma allows us to restrict our attention to type~(b) containers that have at most $\eps n$ elements from $[n/2]$. 

\begin{lemma}\label{lem-CE}
Let $F \in \cF$ be a container of type (b) so that $|F \cap [n/2]|\ge \ep n$. 
Then $f_{\max}(F) \leq 2^{(1/4- \delta)n}$.
\end{lemma}

\begin{proof}
Define $c\ge \ep$ so that $|F \cap [n/2]|=cn$.
Since $F$ is of type (b), $F=A \cup B$ where $|A|\leq \alpha n$ and $B$ is sum-free where $\min (B) \geq 0.47n$. Therefore $cn \leq (0.03+\alpha)n$.

As $|F \cap [n/2]|= cn$, $|B \cap [0.47n, n/2]|\geq (c-\alpha)n$ and so trivially 
 $|(B+B)\cap [0.94n, n]|\ge (2c-4\alpha )n$.
Therefore, since $B$ is sum-free, $F$ is missing at least $(2c-4\alpha) n$ numbers from $[0.94n,n]$. 
Partition $F=F_1\cup F_2$ where $F_1:=F\cap [n/2]$ and $F_2:=F\setminus F_1$. Note that $|F_2|\le (1/2-2c+4\alpha)n$.

The following  observation is a key idea for the proof  of this lemma. Every maximal sum-free subset of $[n]$ in $F$ can be built in the following two steps. First, fix an arbitrary sum-free set $S\subseteq F_1$. Next, extend $S$ in $F_2$ to a maximal one. Since $|F_1|=cn$, there are at most $2^{cn}$ ways to pick $S$. By Lemma~\ref{lem-mis}, the number of choices for the second step is at most the number of maximal independent sets $I$ in $L_S[F_2]$.

\begin{claim}\label{moved}
There are at most $2^{(1/4-\ep/20)n}$ maximal sum-free subsets $M$ of $[n]$ in $F$ such that $|M\cap F_1| \leq cn /4$. 
\end{claim}
\proof
Choose an arbitrary sum-free set $S\subseteq F_1$ such that $|S|\leq cn/4$ (there are at most $cn\binom{cn}{cn/4}/4$ choices for $S$). By Lemma~\ref{free}, $L:=L_S[F_2]$ is triangle-free. So ${\rm MIS}(L)\leq 2^{|F_2|/2} \leq 2^{(1/4-c+2\alpha)n}$
by~\eqref{htnew}. Thus, the number of maximal sum-free subsets of $[n]$ in $F$ with at most $cn/4$ elements from $F_1$ is at most
$$\frac{cn}{4} {cn\choose\frac{cn}{4}}\cdot 2^{(1/4-c+2\alpha)n}\leq 2^{(1/4-c/10+2\alpha)n}\leq 2^{(1/4-\ep /20)n},$$
where the last inequality follows since $\alpha \ll \eps \leq c$.  
\endproof
Let $S \subseteq F_1$ be sum-free such that $|S|> cn/4$.
Claim~\ref{moved} together with our earlier observation implies that to prove the lemma it suffices to show that
 ${\rm{MIS}}(L_S[F_2])\le 2^{(1/4-c-2\delta )n}$. 

By Lemma~\ref{free}, $L_S[F_2]$ is triangle-free.
We may assume that  $F$ is missing at most $(2c+4\delta)n$ numbers from $[0.94n,n]$.
Indeed, otherwise by~\eqref{htnew}, ${\rm{MIS}}(L_S[F_2]) \le 2^{(1/4-c-2\delta)n}$, as required.
\begin{claim}\label{claimy}
We may assume that $(2c-4\alpha)n \le |[n/2+1,n]\setminus F|\le (2c+9 \delta)n$.
\end{claim}
\begin{proof}
Since we already know that $(2c-4\alpha)n \le |[0.94n,n]\setminus F|\le (2c+4 \delta)n$, to prove the claim we only need to prove that $F$ is missing at most $5 \delta n$ elements from $[0.5n,0.94n]$. Suppose to the contrary that $F$ is missing at least $5 \delta n$ numbers from $[0.5n,0.94n]$. Then $|F_2|\le (1/2-2c+4\alpha -5  \delta)n \leq (1/2-2c-4 \delta)n$ and so by~\eqref{htnew}, ${\rm{MIS}}(L_S[F_2])\le 2^{(1/4-c-2 \delta)n}$. 
\end{proof}

\begin{claim}\label{mbound}
Set $m:=\min (S)$. Suppose that $m < (1/4-2c)n$ or $m > (1/4+\eps)n$. Then  ${\rm{MIS}}(L_S[F_2])\le 2^{(1/4-c-2 \delta)n}$.
\end{claim} 
\begin{proof}
Suppose  that $m> (1/4+\ep) n$.
 Then in $L:=L_S[F_2]$ a vertex  $x \in [(3/4-\eps)n,(3/4+\eps) n]=:N$
is either isolated or adjacent only to itself. 
Thus ${\rm {MIS}}(L)= {\rm {MIS}}(L')$ where $L':= L\setminus N$. Recall that $(2c-4\alpha)n \le |[0.94n,n]\setminus F|$.
Hence,~\eqref{htnew} implies that, ${\rm{MIS}}(L)\le 2^{(1/4-c+2\alpha- \eps)n}\leq 2^{(1/4-c-2 \delta)n}$.

Now suppose that $m< (1/4-2c)n$. Then $L:=L_S[F_2]$ contains at least $100  \delta n$ vertex-disjoint copies of $P_3$. Indeed, consider the set of all $P_3$s with vertex set $\{n/2+i,n/2+m+i,n/2+2m+i\}$ for all $1\le i\le n/2-2m$. Since $m\le (1/4-2c)n$, we have at least $n/2-2m\ge 4cn$ such $P_3$s. By Claim~\ref{claimy}, at most $(2c+9 \delta)n$ elements from $[n/2+1,n]$ are not in $F$. Hence, $L$ contains at least $(2c-9 \delta)n\ge 700  \delta n$ of these copies of $P_3$. Note that these copies of $P_3$ may not be vertex-disjoint, but given one of these  copies $P$ of $P_3$, there are at most $6$ copies of $P_3$ of this type that intersect $P$ in $L$.
So $L$ contains a collection of  $100 \delta n$ vertex-disjoint copies of $P_3$.
Using Lemma~\ref{lem-mis-p3}, we have ${\rm{MIS}}(L)\le 2^{(1/4-c+2\alpha)n-4 \delta n}\leq 2^{(1/4-c-2 \delta)n}$. \end{proof}

 By Claim~\ref{mbound}  we may now  assume that  $(1/4-2c)n \leq m \leq (1/4+\eps)n$.

\begin{claim}\label{blb} Set $b:=\min _2 (S)$.
If $b\leq (1/2-4c)n$ then ${\rm{MIS}}(L_S[F_2])\le 2^{(1/4-c-2 \delta )n}$.
\end{claim}
\begin{proof}


We claim that $L:=L_S[F_2]$ contains at least $100 \delta n$ vertex-disjoint copies of $P_3$. Consider the set of all $P_3$s with vertex set $\{n/2+i,n/2+b+i,n/2+b-m+i\}$ for all $1\le i\le n/2-b$. Since $b\le n/2-4cn$, we have at least $n/2-b\ge 4cn$ such $P_3$s. Note that $F$ might be missing up to $(2c+9 \delta)n$ elements from $[n/2+1,n]$. Hence, $L$ contains at least $(2c-9 \delta)n \geq 700  \delta n$ of these copies of $P_3$.  Note that these copies of $P_3$ may not be vertex-disjoint, but given one of these  copies $P$ of $P_3$, there are at most $6$ copies of $P_3$ of this type that intersect $P$ in $L$.
So $L$ contains a collection of  $100 \delta n $ vertex-disjoint copies of $P_3$. Hence,  Lemma~\ref{lem-mis-p3} implies that ${\rm{MIS}}(L_S[F_2])\le 2^{(1/4-c-2 \delta)n}$.
\end{proof}
So now we may assume that $|S|> cn/4$, $(1/4-2c)n \leq m \leq (1/4+\eps)n$  and $b\ge (1/2-4c)n$. Thus, at least $cn/4$ elements from $[(3/4-6c)n,(3/4+\eps)n]$ lie in $S+m$. Every element of $S+m$ is either missing from $F_2$ or has a loop in $L_S [F_2]$. Recall that  $F_2$ is missing $(2c -4\alpha)n$ elements from $[0.94n,n]$. Thus, altogether at least $2cn-4\alpha n +cn/4\geq 2cn +4 \delta n$ elements from $[n/2+1,n]$ are either missing from
 $F_2$ or have a loop in $L_S [F_2]$. Hence, we have,
$${\rm{MIS}}(L_S[F_2])\le 2^{(1/4-c-2 \delta)n}.$$
\end{proof}

\begin{lemma}\label{lem-CE*}
 Let $F \in \cF$ be a container of type (b) so that $|F \cap [n/2]|\leq \eps n$. 
Let $f^* _{\max} (F)$ denote the number of maximal sum-free subsets $M$ of $[n]$ in $F$ that  satisfy at least one of the following properties:

(i) $\min (M) > (1/4+2\eps) n$ or $\min (M) < (1/4-175\eps) n$;

(ii) $\min _2 (M) \leq (1/2-350\eps) n$.

\noindent
Then $f^*_{\max}(F) \leq 2^{(1/4-\eps)n}$.
\end{lemma}

\begin{proof}
Since $F$ is of type (b), $F=A \cup B$ for some $A,B$ where $|A|\leq \alpha n$ and $B$ is sum-free where $\min (B) \geq 0.47n$. Partition $F=F_1\cup F_2$ where $F_1:=F\cap [n/2]$ and $F_2:=F\setminus F_1$.
So $|F_1|\leq \eps n$ by the hypothesis of the lemma. By (\ref{nondec}) we may assume that $F_2=[n/2+1, n]$.

 Every maximal sum-free subset of $[n]$ in $F$ that satisfies (i) or (ii) can be built in the following two steps. First, fix a sum-free set $S\subseteq F_1$. Next, extend $S$ in $F_2$ to a maximal one.
To give an upper bound on the sets $M$  satisfying (i) we choose $S \subseteq F_1$ where $m:=\min (S)$ is such that $m> (1/4+2\ep) n$ or $m < (1/4-175 \eps) n$ (there are at most $2^{|F_1|}\leq2^{\eps n}$ choices
for $S$). Then by arguing similarly to Claim~\ref{mbound} we have that ${\rm MIS} (L_S [F_2]) \leq 2^{(1/4-2\ep) n}$.

To give an upper bound on the sets $M$  satisfying (ii) we choose $S \subseteq F_1$ where $b:=\min_2 (S)$ is such that $b\leq n/2-350\ep n$ (there are at most $2^{|F_1|}\leq 2^{\eps n}$ choices
for $S$). Then by arguing similarly to Claim~\ref{blb} we have that ${\rm MIS} (L_S [F_2]) \leq 2^{(1/4-2\ep) n}$. 

Altogether,  this implies that $f^*_{\max}(F) \leq 2^{(1/4-\eps) n}$ as desired.
\end{proof}

Throughout this subsection, given a maximal sum-free set $M$  we write $m:=\min (M)$ and $b:=\min_2 (M)$ and define $S:=(M\cap [n/2])\setminus \{m\}$.
Lemmas~\ref{lem-CE} and~\ref{lem-CE*} imply that, to count the number of maximal sum-free subsets of $[n]$ lying in type (b) containers, it now suffices to count the number of maximal sum-free sets $M$ with the following structure:

($\alpha$)  $m\in [(1/4-175\eps)n, (1/4+175\eps)n]$.

($\beta$) $b\geq (1/2-350\eps) n$.

\medskip

In particular, the next lemma shows that almost all of the maximal sum-free subsets of $[n]$ that satisfy ($\alpha$) and ($\beta$) lie in type (b) containers only.

\begin{lemma}\label{noovercount}
There are at most $\eps 2^{n/4}$ maximal sum-free subsets of $[n]$ that satisfy ($\alpha$) and ($\beta$) and that lie in type (a) or (c) containers.
\end{lemma}
\proof
By Lemma~\ref{lem-non-ext}, at most $2^{0.249n} \leq \eps 2^{n/4} /2$ such maximal sum-free subsets of $[n]$ lie in type (a) containers.

Suppose that $M$ is a maximal sum-free subset of $[n]$ that satisfies ($\alpha$) and ($\beta$) and lies in a type  (c) container $F$. 
Thus, $F=A\cup B$ where $|A|\leq \alpha n$ and $B \subseteq O$.
Define $F':=B \cap [n/2-350 \eps n, n]$. So,
$|F'|\leq  (1/4+175\eps)n$. By Lemma~\ref{lem-mis}, $M=I \cup S$ where $\min (S)=m$ for some $m\in [(1/4-175\eps)n, (1/4+175\eps)n]$, $(S\setminus \{m\})\subseteq A$ and 
$I$ is a maximal independent set in $G:=L_S [F']$. By the Moon--Moser bound,
$${\rm MIS} (G) \leq 3^{(1/12 +60\eps)n} \leq 2^{(1/4-\eps)n}.$$
In total, there are at most $2^{\alpha n}$ choices for $F$,  at most $350\eps n$ choices for $m$ and at most $2^{\alpha n}$ choices for $S\setminus \{m\}$.
Thus, there are at most
$$2^{\alpha n} \times 350\eps n \times 2^{\alpha n}  \times 2^{n/4-\eps n} \leq \eps 2^{n/4} /2$$
 maximal sum-free subsets of $[n]$ that satisfy ($\alpha$) and ($\beta$) and that lie in type  (c) containers, as desired.
\endproof

For the rest of this subsection, we focus on counting the maximal sum-free sets that satisfy~($\alpha$) and ($\beta$). 
Fix $m,b$ such that $m\in [(1/4-175\eps)n, (1/4+175\eps)n]$ and $b\geq (1/2-350\eps) n$.
Define $t:=|m-n/4|$ and $D:=n/2-b$, so $t,D\leq 350 \eps n$. (Notice that if $b>n/2$, then $D$ is negative.) 
Let $S \subseteq [b,n/2]$ such that $b \in S$, $S \cup \{m\}$ is sum-free and set $s:=|S|\leq D$. In the case when $b>n/2$, we define $S:=\emptyset$.

 Denote by $L:=L(n,m,S)$ the link graph of $S\cup\{m\}$ on vertex set $[n/2+1, n]$. So $L$ is triangle-free by Lemma~\ref{free}. 
We will need the following two bounds on the number of maximal independent sets in $L$.
\begin{lemma}\label{lem-mis-L}
We have the following two bounds on ${\rm{MIS}}(L)$.

(i) ${\rm{MIS}}(L)\le 2^{n/4-D/25}$;

(ii) Let $R$ be defined so that $|S+S|=Rs$. Then ${\rm{MIS}}(L)\le 2^{n/4-(R+1)s/2}$.

\end{lemma}
\begin{proof}
If $D\leq 0$ then (i) follows from (\ref{htnew}). So assume $D>0$.
Notice that there are $D$ vertex-disjoint $P_3$s in $L$: $\{n/2+i, n+i-D, n+i-D-m\}$ for each $1\le i\le D$. (These paths are vertex-disjoint since $D\leq 350 \eps n$ and $m\in [(1/4-175\eps)n, (1/4+175\eps)n]$.) The bound follows immediately from Lemma~\ref{lem-mis-p3}.

For (ii), notice that in $L$ we have  loops at all vertices in $S+S$ and $S+m$ (in total $(R+1)s$ vertices). 
${\rm{MIS}}(L)={\rm{MIS}}(L')$ where $L'$ is the graph obtained from $L$ by deleting all the vertices with loops.
The bound then follows from~\eqref{htnew}.
\end{proof}

The following lemma bounds the number of maximal sum-free sets $M$ satisfying~($\alpha$) and ($\beta$) and with $b$ sufficiently bounded away from $n/2$ from above.

\begin{lemma}\label{lem-int1}
There exists a constant $K=K(\ep)$ such that the number of maximal sum-free sets $M$ in $[n]$ that  satisfy~($\alpha$), ($\beta$) and $b\leq n/2-K$ is at most $\ep 2^{n/4}$.
\end{lemma}
\begin{proof}
Let $K$ be such that $\delta \ll 1/K \ll \ep$. 
Our first claim implies that there are not too many maximal sum-free subsets of $[n]$ with $t$ or $D$ `large'.
\begin{claim}\label{lem-t-D-s} There are at most $\ep 2^{n/4}/5$ maximal sum-free sets $M$ which satisfy~($\alpha$) and ($\beta$) and with
\begin{itemize}
\item[{\rm (a)}] $b\leq n/2-K$;
\item[{\rm(b)}] $t\ge 3D$ or $D\ge 10^9s$.
\end{itemize} 
\end{claim}
\begin{proof}
Fix any  $m,b$ such that $m\in [(1/4-175\eps)n, (1/4+175\eps)n]$ and $n/2-350\eps n \leq b \leq n/2-K$.  Define $t$ and $D$ as before. 
Let $S \subseteq [b,n/2]$ such that $b \in S$, $S \cup \{m\}$ is sum-free and set $s:=|S|\leq D$. Define the link graph $L$ as before. 

Suppose that $t\ge 3D$. If $m=n/4-t$ then for each $i$ with $D+1\le i \le 2t-D$ consider the subgraph  $H_i$ of $L$ induced by
$\{n/2+i, 3n/4+i-t, n+i-2t\}$. Ignoring loops, $H_i$ spans a $P_3$ component in $L$ and so ${\rm MIS} (H_i) \leq 2$.
Indeed, since $t,D\leq 350 \eps n$ and $\min(S)=b=n/2-D$,  the vertex $3n/4+i-t$ has no neighbour in $L$ generated by $S$. Also, since  $n/2+i+b=n+i-D>n$ and $n+i-2t-b=n/2+i-2t+D\le n/2$, neither $n/2+i$ nor $n+i-2t$ has a neighbour generated by $S$ in $L$. Recall $L$ and thus  $L':=L\setminus \cup_{i=D+1}^{2t-D} H_i$ is  triangle-free. Thus by~\eqref{htnew} we have
$${\rm{MIS}}(L)\le {\rm{MIS}}(L')\cdot \prod_i {\rm{MIS}}(H_i)\le 2^{[n/2-3(2t-2D)]/2}\cdot 2^{2t-2D}\le 2^{n/4-(t-D)}\le 2^{n/4-2t/3}.$$

Otherwise $m=n/4+t$ and then there are $2t$ isolated vertices $\{3n/4-t+1,\ldots, 3n/4+t\}$ in $L$. Then by~\eqref{htnew}, ${\rm{MIS}}(L)\le 2^{n/4-t}$.

Given fixed $t$, there are $2$ choices for $m$. There are at most $ 2^{t/3}$ choices for $S$ so that $D\leq t/3$. Further, fixing $S$ determines $b$ and $D$. 
Altogether, this implies that the number of maximal sum-free subsets $M$ of $[n]$ that satisfy ($\alpha$), ($\beta$), (a) and $t\ge 3D$ is at most
\begin{eqnarray}\label{eq-first}
2 \cdot \sum_{t\ge 3D\ge 3K}2^{t/3}\cdot 2^{n/4-2t/3}\le 2 \cdot \sum_{t\ge 3K}2^{n/4-t/3}\le\frac{\ep}{10}\cdot 2^{n/4},
\end{eqnarray}
where the last inequality follows since $1/K \ll \ep$ and $n$ is sufficiently large.

Suppose now that $t\le 3D$ and $D/s\ge 10^9$. For fixed $D \geq K$ there are $3D$ choices for $t$ and so at most $6D\le 2^{2\log D}$ choices for $m$. Given fixed $D$, there are $D=2^{\log D}$ choices for $s$.
For fixed $D,s$ there are ${D\choose s}\le \left(\frac{eD}{s}\right)^s\le 2^{s\log(eD/s)}$ choices for $S$.
 Note that when $D/s\ge 10^9$, $3\log D+s\log(eD/s)\le D/50$. 
Together, with  Lemma~\ref{lem-mis-L}(i), this implies that the number of maximal sum-free subsets $M$ of $[n]$ that satisfy ($\alpha$), ($\beta$), (a) and with $t\le 3D$ and $D/s\ge 10^9$ is at most
\begin{eqnarray}\label{eq-second}
\sum_{D\ge K} 2^{2\log D}\cdot 2^{\log D}\cdot 2^{s\log(eD/s)}\cdot 2^{n/4-D/25}\le \sum_{D\ge K} 2^{n/4-D/50}\le\frac{\ep}{10}\cdot 2^{n/4}.
\end{eqnarray}
\end{proof}

By Claim~\ref{lem-t-D-s}, to complete the proof of the lemma it suffices to count   the number of maximal sum-free subsets $M$ of $[n]$ that satisfy ($\alpha$), ($\beta$) and 

($\gamma_1$) $b\leq n/2-K$;

($\gamma_2$)  $s \geq D/10^9 \geq  K/10^9 $; 

($\gamma_3$)  $t <3D$.

Fix any  $m,b$ such that $m\in [(1/4-175\eps)n, (1/4+175\eps)n]$ and $n/2-350\eps n \leq b \leq n/2-K$.  
Let $S \subseteq [b,n/2]$ such that $b \in S$, $S \cup \{m\}$ is sum-free and set $s:=|S|\leq D$. Define the link graph $L$ as before. 

Choose $s$ and $D$ such that $s \geq D/10^9$. For each fixed $s$ there are at most $10^9s$ choices for $D$.
For a fixed $s \geq D/10^9$, there are at most 
$ 6D\le 10^{10}s\le 2^{2\log s}$ choices for $m$ so that $t <3D$ and at most 
$\binom{10^9s}{s}$ choices for $S$.
So there are at most
\begin{eqnarray}\label{eq-m-S}
 10^9s \cdot 2^{2\log s}\cdot {10^9s\choose s}\le 10^9s \cdot 2^{2\log s} \cdot 2^{s\log(e\cdot 10^9)}\le2^{49s}
\end{eqnarray}
choices for the pair $S,m$ given fixed $s$.
Let $R$ be defined so that $|S+S|=Rs$. We now distinguish two cases depending on the size of $S+S$. 

The number of maximal sum-free subsets $M$ in $[n]$ that satisfy ($\alpha$), ($\beta$), ($\gamma_1$)--($\gamma_3$)  and $R \geq 100$ is at most
\begin{eqnarray}\label{eq-third}
\sum_{s\ge K/10^9} 2^{49s}\cdot 2^{n/4-50s}\le\sum_{s\ge K/10^9} 2^{n/4-s}\le\frac{\ep}{10}\cdot 2^{n/4}.
\end{eqnarray}
(Here we have applied \eqref{eq-m-S} and Lemma~\ref{lem-mis-L} (ii).)

Let $s_0(1/9, 100)$ be the constant returned from Lemma~\ref{lem-G-M}. Since we chose $K$ sufficiently large, we have that  $s\ge K/10^9\ge s_0(1/9, 100)$.

Now suppose $R\leq 100$.
 Then by Lemma~\ref{lem-G-M} the number of choices for $S$ is at most
\begin{eqnarray}\label{eq-S}
 2^{ s/9}{\frac{1}{2}Rs\choose s}D^{\lfloor R+1/9\rfloor}\le 2^{ s/9}\cdot 2^{Rs/2}\cdot 2^{4R\log s}\le 2^{Rs/2+2 s/9}.
\end{eqnarray}
Recall that for a fixed $s$, the number of choices for $m$ is at most $ 2^{2\log s}$. Together with Lemma~\ref{lem-mis-L}(ii) and~\eqref{eq-S}, we have that the number of maximal sum-free subsets $M$ in $[n]$ that satisfy ($\alpha$), ($\beta$), ($\gamma_1$)--($\gamma_3$)  and $R \leq 100$ is at most
\begin{eqnarray}\label{eq-fourth}
& &\sum_{s\ge K/10^9} 2^{2\log s}\cdot 2^{Rs/2+2 s/9}\cdot 2^{n/4-(R+1)s/2}\le \sum_{s\ge K/10^9} 2^{n/4-s/2+ s/3}\nonumber\\
&\le& \sum_{s\ge K/10^9} 2^{n/4-s/6}\le\frac{\ep}{10}\cdot 2^{n/4}.
\end{eqnarray}
Thus by~Claim~\ref{lem-t-D-s},~\eqref{eq-third} and~\eqref{eq-fourth}, we have that the number of maximal sum-free sets that satisfy ($\alpha$), ($\beta$) and $b \leq n/2-K$ is at most $\ep\cdot 2^{n/4}$.
\end{proof}

The following lemma bounds the number of maximal sum-free sets when $t$ is large.

\begin{lemma}\label{lem-tlarge}
There are at most $\ep 2^{n/4}$ maximal sum-free sets in $[n]$ that satisfy ($\alpha$) and ($\beta$) and 
with $|m-n/4|=t$ and $b=n/2-D$ such that $D\le K$ and $t\ge 50K$.
\end{lemma}
\begin{proof}
Let us first assume that $m=n/4+t$. 
If $b \leq n/2$ then let $S \subseteq [b,n/2]$ where $b\in S$. Otherwise let $S=\emptyset$.
Then  in the link graph $L:=L(n,m,S)$, every vertex in $\{3n/4-t+1,3n/4+t\}=:N$ is either isolated or adjacent only to itself. Since $D\le K$, the number of choices for $S$ is at most $2^{K}$. Let $L':=L\setminus N$, then by~\eqref{htnew} the number of maximal sum-free sets in this case is at most
\begin{eqnarray*}
\sum_{t\ge 50K}2^{K}\cdot \mis(L')\le \sum_{t\ge 50K}2^K\cdot 2^{n/4-t}\le \ep 2^{n/4}/2.
\end{eqnarray*}
Otherwise, suppose  $m=n/4-t$. If $b \leq n/2$ then let $S \subseteq [b,n/2]$ where $b\in S$. Otherwise let $S=\emptyset$.
 The link graph $L:=L(n,m,S)$ contains $2t$ vertex-disjoint $P_3$s on the vertex set $\{n/2+i, 3n/4-t+i, n-2t+i\}$ where $1\le i\le 2t$. Then by Lemma~\ref{lem-mis-p3}, the number of maximal sum-free sets in this case is at most
\begin{eqnarray*}
\sum_{t\ge 50K}2^{K}\cdot \mis(L)\le \sum_{t\ge 50K}2^K\cdot 2^{n/4-2t/25}\le \ep 2^{n/4}/2.
\end{eqnarray*}
\end{proof}

By Lemmas~\ref{lem-int1} and~\ref{lem-tlarge}, we now need only focus on maximal sum-free sets with 
\begin{eqnarray}\label{eq-m-S1}
t,D\le 50K, \quad i.e. \quad S\subseteq [n/2-50K, n/2] \quad\mbox{  and } \quad m\in[n/4-50K, n/4+50K],
\end{eqnarray} 
where here $D$ may be negative and $S=\emptyset$.
Given any $m,S$ satisfying~\eqref{eq-m-S1} so that $2m \not \in S$, define $C(n,m,S):=\frac{|{\rm{MIS}}(L(n,m,S))|}{2^{n/4}}$. Notice that not every maximal independent set in $L(n,m,S)$ necessarily gives a maximal sum-free set in $[n]$. This happens exactly when a set $I$ is a maximal independent set in both $L(n,m,S)$ and $L(n,m,S^*)$ for some sum-free $S^*\supset S$ such that $S^* \subseteq [n/2]\setminus \{m,2m\}$. Let $\cI(n,m,S)$ be the set of all maximal independent sets in $L(n,m,S)$ that do not correspond to maximal sum-free sets in $[n]$. For each $I\in\cI(n,m,S)$, define $S^*(I)$ to be a largest sum-free set such that $S \subseteq S^*(I)\subseteq [n/2]\setminus \{m,2m\}$ and $I$ is also a maximal independent set in $L(n,m,S^*(I))$. Further partition $\cI(n,m,S):=\cI_1(n,m,S)\cup \cI_2(n,m,S)$, in which $\cI_1(n,m,S)$ consists of all those $I\in \cI(n,m,S)$ with $S^*(I)\subseteq [n/2-50K,n/2]$. Let $\msf(n,m,S)$ be the number of maximal sum-free sets $M$ in $[n]$ that satisfy ($\alpha$) and ($\beta$) with $\min(M)=m$ and $(M\cap[n/2])\setminus\{m\}=S$. For $i=1,2$, further define $C_i(n,m,S):=\frac{|\cI_i(n,m,S)|}{2^{n/4}}$. Then clearly by the definition we have 
$$\msf(n,m,S)=[C(n,m,S)-C_1(n,m,S)-C_2(n,m,S)]2^{n/4}.$$
Notice that every set $I\in \cI_2(n,m,S)$ is a maximal independent set in $L(n,m,S^*(I))$ with $\min(S^*(I))\le n/2-50K$, it then follows from Lemma~\ref{lem-int1} that $\sum_{m,S:~t,D\le 50K} C_2(n,m,S)\le \ep$. 

Thus, the number of maximal sum-free sets $M$ in $[n]$ that satisfy ($\alpha$) and ($\beta$) is at least 
\begin{eqnarray*}
\sum_{m,S:~t,D\le 50K}\msf(n,m,S)
&=&\sum_{m,S:~t,D\le 50K}[C(n,m,S)-C_1(n,m,S)-C_2(n,m,S)]2^{n/4}\\
&\ge&\sum_{m,S:~t,D\le 50K}[C(n,m,S)-C_1(n,m,S)]2^{n/4}-\ep 2^{n/4}.
\end{eqnarray*}
On the other hand, by Lemmas~\ref{lem-int1} and~\ref{lem-tlarge}, the number of maximal sum-free sets $M$ in $[n]$ that satisfy ($\alpha$) and ($\beta$) is at most
\begin{eqnarray*}
\sum_{m,S}\msf(n,m,S)&=&\sum_{m,S:~t,D\le 50K}\msf(n,m,S)+\sum_{m,S:~\max\{t,D\}>50K}\msf(n,m,S)\\
&\le&\sum_{m,S:~t,D\le 50K}[C(n,m,S)-C_1(n,m,S)]2^{n/4}+2\ep 2^{n/4}.
\end{eqnarray*}
By defining $C(n):=\sum_{m,S:~t,D\le 50K}[C(n,m,S)-C_1(n,m,S)]$, together with Lemmas~\ref{lem-CE},~\ref{lem-CE*} and~\ref{noovercount}, we have that the number of maximal sum-free sets of $[n]$ contained in type (b) containers  is $(C(n)\pm 4\ep)2^{n/4}$.

We now proceed to prove that for any $n'\equiv n \mod 4$, $C(n')=C(n)$. We need the following lemma, which roughly states that for any ``fixed'' choice of $m$ and $S$, the link graphs on $[n/2+1, n]$ and $[n'/2+1, n']$ differ by a component consisting of an induced matching of size $(n'-n)/4$. To be formal, fix $t\in [-50K, 50K]$, $S_0\subseteq [50K]$ and $\ell \in \mathbb N$. Define
\begin{eqnarray}\label{eq-iso}
n':=n+4\ell, \quad m:=n/4-t,\quad m':=n'/4-t,\quad S:=n/2-S_0,\quad S':=n'/2-S_0.
\end{eqnarray}
The proof of the following lemma for the case $m=n/4+t$ and $m'=n'/4+t$ is almost identical except only simpler, we omit it here.
\begin{lemma}\label{lem-iso}
Let $n', m, m', S, S'$ be given as in~\eqref{eq-iso}. Then $L(n', m', S')$ is isomorphic to the disjoint union of $L(n,m,S)$ and a matching of size $\ell$.
\end{lemma}
\begin{proof}
Let $I_1:=[n'/2+200K+1, 3n'/4-200K+t]$ and $I_2:=[3n'/4+200K+1-t, n'-200K]$. Notice first that the induced subgraph of $L':=L(n',m',S')$ on $I_1\cup I_2$ is a matching: $\{n'/2+200K+1, 3n'/4+200K+1-t\}, \ldots, \{3n'/4-200K+t, n'-200K\}$. Let $\cM$ be the first $\ell$ matching edges in $L'[I_1\cup I_2]$, i.e.~$\{n'/2+200K+1, 3n'/4+200K+1-t\}, \ldots, \{n'/2+200K+\ell, 3n'/4+200K+\ell-t\}$. Define $L'':=L'\setminus\cM$. It is a straightforward but tedious task to see that $L''$ is isomorphic to $L:=L(n,m,S)$. We give here only the mapping $f:V(L)\rightarrow V(L'')$ that defines an isomorphism:
\begin{itemize}
\item $[n/2+1, n/2+200K]\quad \rightarrow\quad [n'/2+1, n'/2+200K]$;
\item $[n/2+200K+1, 3n/4+200K-t]\quad \rightarrow\quad [n'/2+200K+\ell+1, 3n'/4+200K-t]$;
\item $[3n/4+200K-t+1, n-200K]\quad \rightarrow\quad [3n'/4+200K+\ell-t+1, n'-200K]$;
\item $[n-200K+1, n]\quad \rightarrow\quad [n'-200K+1, n']$.
\end{itemize}
\end{proof}

Fix $n',m,m',S, S'$ satisfying~\eqref{eq-m-S1} and~\eqref{eq-iso}. By the definition of $C(n)$, to show that $C(n)=C(n')$, it suffices to show that $C(n,m,S)=C(n',m',S')$ and $C_1(n,m,S)=C_1(n,m,S)$. Let $\cM$ and $f$ be the matching of size $\ell$ and the mapping from Lemma~\ref{lem-iso}. As an immediate consequence of Lemma~\ref{lem-iso}, we have
$$C(n',m',S')=\frac{|{\rm{MIS}}(L(n',m',S'))|}{2^{n'/4}}=\frac{|{\rm{MIS}}(L(n,m,S))|\cdot \mis(\cM)}{2^{n/4}\cdot 2^{\ell}}=C(n,m,S).$$
As for $C_1(n,m,S)$, it suffices to show that every $I\in\cI_1(n,m,S)$ corresponds to precisely $2^{\ell}$ sets in $\cI_1(n',m',S')$. Fix an arbitrary $I\in\cI_1(n,m,S)$ and recall that $S\subseteq S^*(I)\subseteq [n/2-50K,n/2]$. Let $S^{**}$ be the ``counterpart'' (as in $S'$ to $S$ in~\eqref{eq-iso}) of $S^*(I)$ in $[n']$, i.e.~$S^{**}:=n'/2-(n/2-S^*(I))\subseteq [n'/2-50K, n'/2]$. By the definition of $\cM$, edges generated by  $S', S^{**}\subseteq [n'/2-50K,n'/2]$ on $[n'/2, n']$ are not incident to any vertex in $\cM$. Hence by adding any maximal independent set of $\cM$ to $f(I)$, we obtain $|\mis(\cM)|=2^{\ell}$ many maximal independent sets $I'$ in $\cI_1(n',m',S')$ with $S^*(I')=S^{**}$ as required. We have concluded the following main result of this subsection.

\begin{lemma}\label{lem-intmain} For each $1\leq i \leq 4$, there is a constant $D_i$ such that,  if $n\equiv i \mod 4$ then
the number of maximal sum-free subsets of $[n]$ in type (b) containers is $(D_i\pm 4\eps ) 2^{n/4}$.
\end{lemma}

We remark that the constants $D_i$ can be efficiently computed. Indeed, from the above argument, we get that $D_i=C(n_0)$ for sufficiently large $n_0$ with $n_0\equiv i$ mod 4. Note that $C(n_0)$ is determined by $O(1)$ many link graphs (the number of such graphs is at most the number of choices for $(m,S)$, which is at most $100K\cdot 2^{50K}$ due to~\eqref{eq-m-S1}). Fix one such graph, say $H_S$, notice crucially that $H_S$ is the disjoint union of some constant-order ($O_K(1)$ vertices) graph $F_S$ and a matching $M$ of size $|M|=n/4+O_K(1)$. Then by definition, $C(n_0)$ is determined solely by $\{F_S\}_{S\subseteq [n/2-50K,n/2]}$. 
We explain the consequences of this regarding computing the constants $C_i$ in Section~\ref{explain}.


\subsection{Type (c) containers}\label{subsec-odd}
The next result implies that the number of maximal sum-free subsets of $[n]$ that contain at least two even numbers and that lie in type (c) containers is `small'.

\begin{lemma}\label{lem-odd} 
Let $F\in \cF$ be  a container of type (c). Then $F$ contains at most $2^{(1/4-\ep /2) n}$ maximal sum-free subsets of $[n]$ that contain at least two even numbers.
\end{lemma}
\proof
Let $F\in\cF$ be as in the statement of the lemma. Let $K$ be a sufficiently large constant so that 
\begin{align}\label{choosek}
 \sum_{0 \le i\le n/K}{n\choose i}3^{{\frac{5n}{36}} + \frac{n}{3K}}  \leq 2^{0.249n}.
\end{align}
Since $1/n \ll \ep \ll 1$, we have that $ \ep \ll 1/K^2$.
By (\ref{nondec}), we may assume that $F=O\cup C$ with $C\subseteq E$ and $|C|\leq \alpha n$. Similarly as before, every maximal sum-free subset of $[n]$ in $F$ can be built from choosing a sum-free set $S\subseteq C$ (at most $2^{|C|}\le 2^{\alpha n}$ choices) and extending $S$ in $O$ to a maximal one. Fix an arbitrary sum-free set $S$ in $C$ where $|S|\geq 2$ and let $G:=L_S[O]$ be the link graph of $S$ on vertex set $O$. Since $O$ is sum-free and $\alpha \ll \eps$, Lemma~\ref{lem-mis} implies that, to prove the lemma,
it suffices to show that ${\rm{MIS}}(G)\leq 2^{(1/4-\ep)n}$. We will achieve this in two cases depending on the size of $S$.

\noindent\textbf{Case 1:} $|S|\geq 2 K^2$.

In this case, we will show that $G$ is `not too sparse and almost regular'. Then we apply Lemma~\ref{lem-ik}.

We first show that $\de(G)\ge |S|/2$ and $\De(G)\le 2|S|+2$, thus $\De(G)\le 5\de(G)$. Let $x$ be any vertex in $O$. 
If $s \in S$ such that $s<\max\{x,n-x\}$ then at least one of $x-s$ and $x+s$ is adjacent to $x$ in $G$. If $s \in S$ 
such that $s\geq \max\{x,n-x\}$ then $s-x$ is adjacent to $x$ in $G$. By considering all $s \in S$ this implies that $\deg_G (x) \geq |S|/2$ (we divide by $2$ here as an edge $xy$ may arise from two different elements of $S$).
For the upper bound consider $x \in O$. If $xy \in E(G)$ then $y=x+s$, $x-s$ or $s-x$ for some $s \in S$ and only two of these terms are positive. Further, there may be a loop at $x$ in $G$ (contributing $2$ to the degree of $x$ in $G$). Thus, $\deg_G (x) \leq 2|S|+2$, as desired.

Note that $\de(G)^{1/2}\ge K$.  Thus, applying Lemma~\ref{lem-ik} to $G$ with $k=5$ we obtain that
$${\rm{MIS}}(G)\le \sum_{0 \le i\le n/K}{n\choose i}3^{{\frac{5n}{36}} + \frac{n}{3K}}  \stackrel{(\ref{choosek})}{\leq} 2^{0.249n}.$$

\noindent\textbf{Case 2:} $2\leq |S|\leq 2 K^2$.

As in Case~1 we have that $\Delta (G) \leq 2|S|+2\leq 5K^2$. Additionally, we need to count triangles in $G$.
\begin{claim}
$G$ contains at most $24|S|^3$ triangles.
\end{claim}
The claim is shown in the proof of Lemma 3.4 in~\cite{BLST}, so we omit the proof here. Let $T\subseteq V(G)$ such that $|T| \leq 24|S|^3$ and $G\setminus T$ is triangle-free.

Let $G_1$ denote the graph obtained from $G$ by removing all loops. Given any $x\in O$ and $s \in S$, one of $x-s, s-x$ is adjacent to $x$ in $G$.  In particular, if $2x\not = s$, then  one of $x-s, s-x$ is adjacent to $x$ in $G_1$. Therefore each $s \in S$ gives arise to at least $(|O|-1)/2$ edges in $G_1$. Given distinct $s,s' \in S$, there is at most one pair $x,y \in O$ such that $s,x,y$ and $s',x,y$ are both Schur triples. Thus, 
since $|S|\geq 2$,  this implies that $e(G_1)\geq |O|-2$.
Set $G':= G_1 \setminus T$. Note that $\Delta (G_1) \leq 5K^2$, $|G'| \leq |O|$ and $e(G')\geq |O|-2 -|T|5K^2 \geq 3|O|/4$.
Thus Corollary~\ref{use} implies that ${\rm MIS} (G_1) \leq 2^{(1/4-\ep) n}.$ Fact~\ref{newfact} therefore implies that
${\rm MIS}(G) \leq 2^{(1/4-\ep) n}$, as desired.
\endproof

Note that the argument in Case~2 of Lemma~\ref{lem-odd} immediately implies the following result.
\begin{lemma}\label{2bound}
 Given any distinct $x,x' \in E$, 
$${\rm MIS}(L_{ \{x,x'\} }[O]) \leq 2^{(1/4-\ep) n}.$$ 
\end{lemma}

Given $n \in \mathbb N$, let $f' _{\max} (n)$ denote the number of maximal sum-free subsets of $[n]$ that contain precisely one even number. The next result implies that  $f' _{\max} (n)$ is approximately equal to the  number of maximal independent sets in the link graphs $L_x [O]$ where $x \in E$.

\begin{lemma}\label{boundingf'}
\begin{align}\label{1bound}
\sum _{x\in E} {\rm MIS} (L_x [O]) - 2 \cdot \sum _{x \not = x'\in E} {\rm MIS} (L_{ \{x,x'\} }[O]) \leq f' _{\max} (n) \leq \sum _{x\in E} {\rm MIS} (L_x [O]) .
\end{align}
In particular, 
\begin{align}\label{1bound2}
\sum _{x\in E} {\rm MIS} (L_x [O]) - 2^{(1/4-\ep /2)n} \leq f' _{\max} (n) \leq \sum _{x\in E} {\rm MIS} (L_x [O]) .
\end{align} 
\end{lemma}
\proof
Given any maximal sum-free subset $M$ of $[n]$ that contains precisely one even number $x$, $M\setminus \{x\}$ is a maximal independent set in $L_x [O]$. So the upper bound in (\ref{1bound}) follows.

\begin{claim}\label{onlyeven}
Suppose $x \in E$ and $S$ is a maximal independent set in $L_x[O]$. Let $M$ denote the maximal sum-free subset of $[n]$ that contains $S \cup \{x\}$. Then $M \setminus S \subseteq E$.
\end{claim}
\proof
Suppose not. Then there exists $S' \subseteq M$ such that $S \subset S' \subseteq O$. But as $M$ is sum-free, $S'$ is an independent set in $L_x [O]$, a contradiction to the maximality of $S$.
\endproof
Suppose $y \in E$ and $S$ is a maximal independent set in $L_y[O]$. If $S \cup \{y\}$ is not a maximal sum-free subset of $[n]$ then Claim~\ref{onlyeven} implies that there exists $y' \in E \setminus \{y\}$ such that $S\cup \{y,y'\}$ is sum-free. In particular, $S$ is a maximal independent set in $L_{\{y,y'\}} [O]$. In total there are at most
$$2\cdot \sum _{x \not = x'\in E} {\rm MIS} (L_{ \{x,x'\} }[O])$$
such pairs $S,y$. Thus, the lower bound in (\ref{1bound}) follows.

The lower bound in (\ref{1bound2}) follows since, by Lemma~\ref{2bound},
$$2\cdot \sum _{x \not = x'\in E} {\rm MIS} (L_{ \{x,x'\} }[O]) \leq  2n^2 \cdot 2^{(1/4-\ep) n}\leq  2^{(1/4-\ep/2) n},$$ where the last inequality follows since $n$ is sufficiently large.
\endproof

The next result determines $\sum _{x\in E} {\rm MIS} (L_x [O])$ asymptotically and thus, together with Lemma~\ref{boundingf'} determines, asymptotically, $f'_{\max} (n)$.

\begin{lemma}\label{lem:count}
Given $1 \leq i \leq 4$, there exists a constant $D'_i $ such that, if $n\equiv i \mod 4$,
$$(D'_i - \eps ) 2^{n/4} \leq \sum _{x\in E} {\rm MIS} (L_x [O])\leq  (D'_i + \eps ) 2^{n/4}.$$
\end{lemma}
\proof
Suppose that $n\equiv 0 \mod 4$. The proofs for the other cases are essentially identical, so we omit them. Let $2n/3 < m \leq n$ be even.
Consider $G:=L_m [O]$. The edge set of $G$ consists of precisely the following edges:
\begin{itemize}
\item An edge between $i$ and $m-i$ for every odd $i< m/2$;
\item A loop at $m/2$ if $m/2$ is odd;
\item An edge between $i$ and $m+i$ for all odd $i \leq n-m <n/3$.
\end{itemize}
In particular, since $m >2n/3$, if $i <m/2$ is odd then in $G$, $m-i$ is only adjacent to $i$. Altogether this implies that if $m/2$ is even then $G$ is the disjoint union of:
\begin{itemize}
\item $(n-m)/2$ copies of $P_3$;
\item A matching containing $(3m-2n)/4$ edges.
\end{itemize}
In this case ${\rm MIS} (G) =2^{(n-m)/2} \times 2^{(3m-2n)/4} =2^{m/4}$.
If $m/2$ is odd then $G$ is the disjoint union of:
\begin{itemize}
\item $(n-m)/2$ copies of $P_3$;
\item A single loop;
\item A matching containing $(3m-2n-2)/4$ edges.
\end{itemize}
In this case ${\rm MIS} (G) =2^{(m-2)/4}$.

Thus, 
\begin{align}\label{summy}
\sum _{m\in E  \, : \, m >2n/3} {\rm MIS} (L_m [O]) & \leq \sum ^{n} _{m=4 \, : \, m \equiv 0 \, {\rm mod } \, 4} 2^{m/4} + \sum ^{n} _{m=2 \, : \, m \equiv 2 \, {\rm mod}\, 4} 2^{(m-2)/4} \nonumber \\
& = \sum _{m=1} ^{n/4} 2^m + \sum _{m=0} ^{n/4-1} 2^m \leq (3 +\ep /2)2^{n/4}.
\end{align}
Further,
\begin{align}\label{summy1}
\sum _{m\in E  \, : \, m >2n/3} {\rm MIS} (L_m [O]) \geq (3-\ep /2)2^{n/4} - \sum ^{2n/3} _{m=1} 2^{m/4} \geq (3-\eps)2^{n/4}.
\end{align}

Consider $m \in E$ where $m \leq 2n/3$ and set $G:=L_m [O]$. It is easy to see that $G$ is the disjoint union of paths that contain at least $3$ vertices and in the case when $m/2$ is odd, an additional path of length at least $2$ which contains a vertex (namely $m/2$) with a loop. Every such graph on $ n/2$ vertices contains at least $n/10-1$ vertex-disjoint copies of $P_3$.
Therefore, by Lemma~\ref{lem-mis-p3} we have that
\begin{align}\label{summy2}
\sum _{m\in E  \, : \, m \leq 2n/3} {\rm MIS} (L_m [O]) \leq n 2^{n/4- n/250+1}.
\end{align}
Overall, we have that
$$(3-\eps )2^{n/4} \stackrel{(\ref{summy1})}{\leq} \sum _{x\in E} {\rm MIS} (L_x [O]) \stackrel{(\ref{summy}),(\ref{summy2})}{\leq} (3 +\eps /2)2^{n/4} + n 2^{n/4- n/250+1} \leq (3 +\eps )2^{n/4},$$
as desired.
\endproof
We showed that the constant $D'_4$ in Lemma~\ref{lem:count} is equal to $3$. By following the argument given in the proof, it is easy to see that 
\begin{equation}\label{eq-Di-prime}
   D'_1= 3\cdot 2^{-1/4},\quad  D'_2= 2^{3/2}, \quad D'_3 =2^{5/4}, \quad \mbox{ and } D'_4=3.
\end{equation}

The next lemma shows that almost all of the maximal sum-free subsets of $[n]$ that contain precisely one even number lie in type (c) containers only.
\begin{lemma}\label{noovercount1}
There are at most $\eps 2^{n/4}$ maximal sum-free subsets of $[n]$ that contain precisely one even number and that lie in type (a) or (b) containers.
\end{lemma}
\proof
By Lemma~\ref{lem-non-ext}, at most $2^{0.249n} \leq \eps 2^{n/4} /2$ such maximal sum-free subsets of $[n]$ lie in type (a) containers.

Suppose that $M$ is a maximal sum-free subset of $[n]$ that lies in a type  (b) container $F$ and only contains one even number. Define $F':=F \cap O$. Since $F$ is of type (b),
$|F'|\leq (0.53n)/2+\alpha n \leq 0.27n$. By Lemma~\ref{lem-mis}, $M=I \cup \{m\}$ where $m$ is even and $I$ is a maximal independent set in $G:=L_m [F']$. By the Moon--Moser bound,
$${\rm MIS} (G) \leq 3^{0.09n} \leq 2^{(1/4-\eps)n}.$$
In total, there are at most $2^{\alpha n}$ choices for $F$ and at most $ n/2$ choices for $m$.
Thus, there are at most
$$2^{\alpha n} \times \frac{ n}{2} \times 2^{n/4-\eps n} \leq \eps 2^{n/4} /2$$
 maximal sum-free subsets of $[n]$ that that lie in  type  (b) containers  and only contain one even number, as desired.
\endproof

Notice that this completes the proof of Theorem~\ref{thm-main}. Indeed, for each $1 \leq i \leq 4$, set $C_i:=D_i+D'_i$. Lemmas~\ref{lem-non-ext},~\ref{lem-intmain},~\ref{lem-odd},~\ref{boundingf'},~\ref{lem:count} and~\ref{noovercount1} together imply that if $n \equiv i \mod 4$, then 
$$(C_i - \eta )2^{n/4} \leq f_{\max} (n) \leq (C_i + \eta )2^{n/4},$$ as desired.

\subsection{Bounds on the constants $C_i$ in Theorem~\ref{thm-main}}\label{explain}
In the proof of Theorem~~\ref{thm-main} we hid one slight subtlety:
indeed, in equation (\ref{ultimateaim}) the constant $C_i$ actually depends on $\eta$ as well as $i$.
So in the proof of Theorem~\ref{thm-main} what we have  shown is given any $\eta>0$, 
 there is a constant $C_{i,\eta}$ (i.e.~dependent on $i$ and $\eta$) such that if $n$ is sufficiently large and $n \equiv i \mod 4$ then
\begin{align*}
(C_{i,\eta} - \eta )2^{n/4} \leq f_{\max} (n) \leq (C_{i,\eta} + \eta )2^{n/4}.
\end{align*}
This immediately implies the existence of the desired $C_i$ in the statement of the theorem (i.e.~$C_i$ is the limit of the $C_{i,\eta}$ as $\eta \rightarrow 0$).

In the proof we have that  $C_{i,\eta}=D_{i,\eta}+D'_{i,\eta}$ where now $D_{i,\eta}$ is playing the role of what was $D_i$ and $D'_{i,\eta}$ plays the role of $D'_{i}$.
Equation (\ref{eq-Di-prime}) gives the precise values of the $D'_{i,\eta}$ (these only depend on $i$ not $\eta$).
As mentioned after Lemma~\ref{lem-intmain}, one can efficiently determine the value of $D_{i,\eta}$. The time taken depends on $K$, which itself depends on $\eps$ and thus $\eta$ (recall the definition of $\eps$ depends only on $\eta$).

Altogether this implies one can determine $C_{i,\eta}$ in constant time (i.e.~only depending on $\eta$). Since $C_i$ is the limit of the $C_{i,\eta}$ as $\eta \rightarrow 0$, this 
implies
$C_i$ can also be computed up to any additive error (say $\eta '$) in constant time (i.e.~depending only on $\eta '$).


\section{Maximal sum-free sets in abelian groups}\label{group}
Throughout this section, unless otherwise specified, $G$ will be an abelian group of order $n$ and we denote by $\mu(G)$ the size of the largest sum-free subset of $G$. Denote by $f(G)$ the number of sum-free subsets of $G$ and by $\fm(G)$ the number of maximal sum-free subsets of $G$. Given a set $F \subseteq G$, we write 
$\fm (F)$ for the number of maximal sum-free subsets of $G$ that lie
in $F$.

The study of sum-free sets in abelian groups dates back to the 1960s. Although Diananda and Yap~\cite{yan} determined $\mu (G)$ for a large class of abelian groups $G$, it was not until 2005 that Green and Ruzsa~\cite{GR-g} determined $\mu(G)$ for all such $G$. In particular, for every finite abelian group $G$, $2n/7 \leq \mu (G) \leq n/2$.
Further, Green and Ruzsa~\cite{GR-g} determined $f(G)$ up to an error term in the exponent for all $G$, showing that $f(G)=2^{(1+o(1))\mu(G)}$. 

Given $G$, what can we say about $\fm(G)$? Is it also the case that $\fm(G)$ is exponentially smaller than $f(G)$? 
Wolfovitz~\cite{wolf}  proved that $\fm(G)\le 2^{0.406n+o(n)}$ for every finite  group $G$.
For even order abelian groups $G$ this answers the second question in the affirmative since $\mu (G)=n/2$ for such groups.

Our next result strengthens the result of Wolfovitz for abelian groups, and implies that indeed $\fm(G)$ is exponentially smaller than $f(G)$ for all finite abelian groups $G$. Let $G$ be fixed.
By a container lemma~\cite[Proposition~2.1]{GR-g} and a removal lemma~\cite[Theorem~1.4]{G-R} for abelian groups, there exists a collection of
containers $\cF$ such that:

(i) $|\cF|=2^{o(n)}$ and $F \subseteq G$ for all $F \in \cF$;

(ii) Given any $F \in \cF$, $F=B\cup C$ where $B$ is sum-free with size $|B|\leq \mu (G)$ and $|C|=o(n)$;

(iii) Given any sum-free subset $S$ of $G$, there is an $F \in \mathcal F$ such that $S \subseteq F$.

\noindent 
Given sets $S,T \subseteq G$, we can define the link graph $L_S[T]$ analogously to the integer case. In particular, it is easy to check that an analogue of Lemma~\ref{lem-mis} holds for such link graphs.

Let $F \in \cF$ be fixed.
Every maximal sum-free subset of $G$ contained in $F$ can be chosen by picking a sum-free set $S$ in $C$ (at most $2^{o(n)}$ choices by (ii)), and extending it in $B$ (at most $\mis(L_S[B])\le 3^{|B|/3}\leq 3^{\mu(G)/3}$ choices by Lemma~\ref{lem-mis} for abelian groups and the Moon-Moser theorem). Therefore, together this implies the following result.
\begin{prop}\label{easyb} Let $G$ be an abelian group of order $n$.
Then
\begin{align}\label{eq-g}
\fm(G)\le 3^{\mu(G)/3+o(n)}.
\end{align}
\end{prop}
We do not know how far from tight the bound in Proposition~\ref{easyb} is.  In particular, it would be interesting to establish whether the following bound holds.
\begin{ques}\label{q2}
Given an abelian group $G$ of order $n$, is it true that $f_{\max}(G)\le 2^{\mu(G)/2+o(n)}$?
\end{ques}

Let $Z_p^k:=Z_p\otimes Z_p\otimes\cdots\otimes Z_p$. For the group $Z_2^k$, the answer to the above question is affirmative and the upper bound is essentially tight.
\begin{prop}\label{thm-z2}
The number of maximal sum-free subsets of $Z_2^k$ is $2^{(1+o(1))\mu(Z_2^k)/2}$.
\end{prop}
\begin{proof}
Let $n:=|Z_2^k|$. It is known that $\mu(Z_2^k)=n/2$. We first give a lower bound $\fm(Z_2^k)\ge 2^{n/4}$. Write $Z_2^k=Z_2\otimes Z_2\otimes H$,  where $H:=Z_2^{k-2}$. Let $x:=(0,1,0_H)$ and $U:=\{1\}\otimes Z_2\otimes H$. Notice that the link graph $L_x[U]$ is a perfect matching. Indeed, for any vertex $y=(1,a,h)\in U$, all of its possible neighbours in $U$ are $x+y=(1,1+a,h),$ $x-y=(1,1-a,-h)$ and $y-x=(1,a-1,h)$ and these elements of $Z_2^k$ are identical. To build a collection of sum-free subsets, we first pick $x$ and then pick exactly one of the endpoints of each edge in $L_x[U]$. Since $|U|=n/2$, we obtain $2^{n/4}$ sum-free subsets $S$ in this way. These sets might not be maximal, but no further elements from $U$ can be added into any of these sets. Hence distinct $S$ lie in distinct maximal sum-free subsets. Therefore we have 
$$\fm(Z_2^k)\ge 2^{n/4}.$$

We now proceed with the proof of the upper bound. Let $\cF$ be the family of $2^{o(n)}$ containers defined before Proposition~\ref{easyb}.
It suffices to show that $f_{\rm max}(F) \leq 2^{(1/4+o(1))n}$ for every container $F\in\cF$. Fix a container $F\in\cF$. We have $F=B\cup C$ with $B$ sum-free, $|B|\leq \mu (Z_2^k) =n/2$ and $|C|=o(n)$. Every maximal sum-free subset of $Z_2^k$ in $F$ can be built by choosing a sum-free set $S$ in $C$ and extending $S$ in $B$ to a maximal one. The number of choices for $S$ is at most $2^{|C|}=2^{o(n)}$. For a fixed $S$, let ${ \Gamma}:=L_S[B]$ be the link graph of $S$ on $B$. Then Lemma~\ref{lem-mis} (for abelian groups) implies that the number of extensions is at most $\mis(\Gamma)$. 
Observe that $\Ga$ is triangle-free. 
Indeed, suppose to the contrary that there exists  a triangle on vertices $a,b,c\in B\subseteq Z_2^k$. Since for any $x\in Z_2^k$, $x=-x$, we may assume that $a+b=s_1$, $b+c=s_2$ and $a+c=s_3$ for some $s_1,s_2,s_3\in S$. Furthermore, $s_1,s_2,s_3$ are distinct elements in $S$ since $a,b,c$ are distinct in $B$. Then we have $s_1+s_2=a+2b+c=a+c=s_3$,
contradicting  $S$ being sum-free. Thus by~\eqref{htnew}, we have
$$\mis(\Ga)\le 2^{|B|/2}\le 2^{n/4}$$ 
and so
$$\fm (F) \leq 2^{|C|} \cdot 2^{n/4}=2^{(1/4+o(1))n},$$ as desired.
\end{proof}

 The following construction gives a lower bound $\fm(Z_n)\ge 6^{(1/18-o(1))n}$. Let $n=9k+i$ for some $0\le i\le 8$ and $M:=[3k+1, 6k]$.  Set $\Gamma :=L_{\{k,-2k\}}[M]$. Then $|M|/6-o(n)$ components of $\Gamma$ are  copies of $K_3\square K_2$ as there are at most a constant number of components of $\Gamma$ that are not copies of $K_3\square K_2$.  Observe that $K_3\square K_2$ contains  $6$ maximal independent sets. Thus, ${\rm MIS}(\Gamma)\ge 6^{(1/18-o(1))n} $, yielding the desired lower bound on $\fm (Z_n)$. It is known that $\mu(Z_p)=(1/3+o(1))p$, if $p$ is \emph{prime}, so together with~\eqref{eq-g}, we obtain the following result.
 \begin{prop}\label{zp} If $p$ is prime then
$$1.1^{p-o(p)}\le 6^{(1/18-o(1))p}\le \fm(Z_p)\le 3^{(1/9+o(1))p}\le 1.13^{p+o(p)}.$$
\end{prop}
It would be interesting to close the gap in Proposition~\ref{zp}.

We end this section with two more constructions that would match the upper bound in Question~\ref{q2} if it is true. 
For this, we need the following simple fact.
\begin{fact}\label{basicgroup}
Suppose $G$ is an abelian group of odd order. Then given a fixed $x \in G$, there
is a unique solution in $G$ to the equation $2y=x$.
\end{fact}
Notice that Fact~\ref{basicgroup} is false for abelian groups of even order.

\begin{prop}
 Suppose that $3|n$ where $n$ is not divisible by a prime $p$ with $p \equiv 2 \mod \, 3$. Then $f_{\max}(G)\ge 2^{(n-9)/6}=2^{(\mu(G)-3)/2}$.
\end{prop}
\proof
First note that $\mu(G)=n/3$ for such groups (see~\cite{GR-g}).
Let $H\le G$ be a subgroup of index 3. Then there are three cosets $0+H, 1+H, 2+H$. Pick some $x\in 2+H$. Then consider the link graph $\Gamma :=L_x[1+H]$ on $n/3$ vertices.
There is a loop at $2x \in V(\Gamma)$. 
For every $y\in 1+H$, $x+y\in 0+H$, $y-x\in 2+H$ and $x-y \in 1+H$. So $y$ has only one neighbour $x-y$ in $1+H$ (unless $y=2x$, which has a loop). By Fact~\ref{basicgroup}, there is a unique $y \in 1+H$ such that $x-y=y$. Overall this implies that $\Gamma$ consists  of the disjoint union of a matching $M$ of size $(n-3)/6$, with a loop at at most one of the vertices in $M$, together with an additional vertex with a loop.
Clearly ${\rm MIS} (\Gamma) \geq 2^{(n-9)/6}$ and so
$f_{\max}(G)\ge 2^{(n-9)/6}$.
\endproof


\begin{prop} Let $G=Z_7^k$. Then $f_{\max}(G)\ge 2^{n/7-1}=2^{\mu(G)/2-1}$. 
\end{prop}
\proof
First note that $\mu(G)=2n/7$ for such groups (see~\cite{GR-g}).
Let $H\le G$ be a subgroup of index 7. Then pick some $x\in 1+H$. Consider the link graph $\Gamma :=L_x[(2+H)\cup ( 3+H)]$ on $2n/7$ vertices. There is a loop at $2x \in 2+H$ in $\Gamma$.
The remaining edges of $\Gamma$ form a 
 perfect matching between $2+H$ and $3+H$. Therefore ${\rm MIS} (\Gamma) = 2^{n/7-1}$ and so
$f_{\max}(G)\ge 2^{n/7-1}$.
\endproof

We conclude the section with two conjectures.
\begin{conjecture}\label{conjnew1}
For every abelian group $G$ of order $n$,
$$2^{n/7} \leq f_{\max} (G) \leq 2^{n/4+o(n)},$$
where the bounds, if true, are best possible.
\end{conjecture}

We also suspect that there is an infinite class of finite abelian groups for which the upper bounds in Conjecture~\ref{conjnew1} and Question~\ref{q2} are far from tight.
\begin{conjecture}\label{conjnew2}
There is a sequence of finite abelian groups $\{G_i\}$ of increasing order such that for all $i$,
$$f_{\max} (G_i) \leq 2^{\mu (G_i) /2.01}.$$
\end{conjecture}


\section*{Acknowledgements}
The authors are grateful to the BRIDGE strategic alliance between the University of Birmingham and the University of Illinois at Urbana-Champaign. This research was conducted as part of the `Building Bridges in Mathematics' BRIDGE Seed Fund project.

The authors are also grateful to the referees for their careful reviews.

  \begin{appendix}
 \section{Appendix}

   \setcounter{theorem}{0}
    \renewcommand{\thelemma}{\Alph{section}\arabic{theorem}}
		
Here we give the proofs of Lemma~\ref{dense} and Corollary~\ref{use}. 
The following simple facts will be used in the proof of Lemma~\ref{dense}.

\begin{fact}\label{fact1} Suppose that $G$ is a graph. For any maximal independent set $I$ in $G$ that contains $x$, $I\setminus \{x\}$ is a maximal independent set in $G\setminus (N_G(x)\cup \{x\})$.
\end{fact}
Given $x\in V(G)$, let ${\rm MIS}_G (x)$ denote the number of maximal independent sets in $G$ that contain $x$. 

\begin{fact}\label{fact2} Suppose that $G$ is a graph. Given any $x \in V(G)$,
$${\rm MIS}(G) \leq {\rm MIS}_G (x) + \sum _{v \in N_G(x)} {\rm MIS} _G (v).$$
\end{fact}
Notice that Fact~\ref{fact2} is not true in general if $G$ is a graph with loops.

\begin{lemma}[F\"uredi~\cite{fur}]\label{fact3} For $m \geq 6$,
${\rm MIS} (C_m)={\rm MIS}(C_{m-2}) +{\rm MIS} (C_{m-3})$.
\end{lemma}
Lemma~\ref{fact3} implies the following simple result.
\begin{lemma}\label{easylem}
For all $m \geq 4$, ${\rm MIS} (C_m) < 2^{0.49m}$.
\end{lemma}
\proof
It is easy to check that the lemma holds for $m=4,5,6$.
For $m \geq 7$, by induction, Lemma~\ref{fact3} implies that 
$${\rm MIS}(C_m)={\rm MIS}(C_{m-2}) +{\rm MIS} (C_{m-3}) < 2^{0.49m} (2^{-0.98}+2^{-1.47}) <2^{0.49m}.$$
\endproof

\begin{cor}\label{cycle}
If $G$ is the vertex-disjoint union of cycles of length at least $4$ then ${\rm MIS}(G) < 2^{0.49|G|}$.
\end{cor}
We now combine the previous results to prove Lemma~\ref{dense}.

\noindent
{\bf Proof of Lemma~\ref{dense}.}
We proceed by induction on $n$. The case when $n \leq 4$ is an easy calculation. We split the argument into several cases.

\noindent
\textbf{Case 1:} There is a vertex $x\in V(G)$ of degree $0$.\\
By induction $G':=G\setminus \{x\}$ is such that ${\rm MIS}(G')\leq 2^{(n-1)/2-k/(100D^2)}$ and clearly ${\rm MIS}(G)={\rm MIS}(G')$.

\smallskip

\noindent 
\textbf{Case 2:} There is a vertex $x\in V(G)$ of degree $1$.\\
First suppose that $x$ is adjacent to a vertex $y$ of degree $1$. Then consider $G':=G\setminus\{x,y\}$. Note that ${\rm MIS}(G)=2 \cdot {\rm MIS} (G')$. Further,
$|G'|=n-2$, $e(G') \geq (n-2)/2+k$ and $\Delta (G') \leq D$. Thus,  by induction we have that
$${\rm MIS} (G)=2 \cdot {\rm MIS}(G') \leq 2 \times 2^{(n-2)/2-k/(100D^2)} = 2^{n/2-k/(100D^2)},$$
as desired.

Otherwise $x$ is adjacent to a vertex $y$ of degree $d\geq 2$. Consider $G':=G\setminus \{x,y\}$. So $|G'|=n-2$, $e(G') \geq (n-2)/2+k-d+1$ and $\Delta (G') \leq D$. Therefore by induction and Fact~\ref{fact1},
\begin{align}\label{eq1}
{\rm MIS} _G (x) \leq {\rm MIS} (G') \leq 2^{(n-2)/2 -(k-d+1)/(100D^2)} \leq 2^{n/2-k/(100D^2)} (2^{-1+d/(100D^2)}).
\end{align}
Consider $G'':=G\setminus (N_G(y) \cup \{y\})$. So $|G''|=n-d-1$, $e(G'') \geq n/2+k-(d-1)D-1 \geq (n-d-1)/2+(k-(d-1)D)$ and $\Delta (G'') \leq D$. 
Thus, by induction and Fact~\ref{fact1},
\begin{align}
{\rm MIS} _G (y) \leq {\rm MIS} (G'') & \leq 2^{(n-d-1)/2 -(k-(d-1)D)/(100D^2)} \nonumber \\ & =
2^{n/2-k/(100D^2)} (2^{-(d+1)/2+(d-1)/100D}).\label{eq2}
\end{align}
Now as $2 \leq d\leq D$ we have that
$$ 2^{-1+d/(100D^2)} + 2^{-(d+1)/2+(d-1)/100D} \leq 2^{-1+1/100}+ 2^{-3/2+1/100} <1. $$
So (\ref{eq1}) and (\ref{eq2}) together with Fact~\ref{fact2} imply that
$$ {\rm MIS} (G)  \leq {\rm MIS} _G (x) + {\rm MIS} _G (y) < 2^{n/2-k/(100D^2)},$$ as desired.

\smallskip

\noindent 
\textbf{Case 3:} $\delta (G) \geq 4$.\\
Let $v \in V(G)$ be the vertex of smallest degree in $G$ and write $\deg_G(v)=i-1\geq 4$. Given any $w \in N_G(v) \cup \{v\}$ let $G':=G\setminus(N_G(w) \cup \{w\})$. So 
$|G'|=n-\deg_G(w)-1$, $e(G') \geq n/2+(k-\deg_G(w)D)\geq |G'|/2 + (k-\deg_G(w)D)$ and $\Delta (G') \leq D$. Hence by induction and Fact~\ref{fact1}
$${\rm MIS} _G (w)\leq {\rm MIS}(G') \leq 2^{(n-\deg_G(w)-1)/2-(k-\deg_G(w)D)/100D^2)} \leq 2^{(n-i)/2 -(k-iD)/(100D^2)}.$$
Thus by Fact~\ref{fact2} we have that
$${\rm MIS} (G) \leq i \times 2^{(n-i)/2 -(k-iD)/(100D^2)} \leq (i2^{-i/2+i/100}) 2^{n/2-k/(100D^2)} < 2^{n/2-k/(100D^2)},$$ as desired.
(Here we used that for $i\geq 5$, $i2^{-i/2+i/100}<1$.)

\smallskip

\noindent 
\textbf{Case 4:} $\delta (G) = 2$ and there exist  $v,w\in V(G)$ such that $\deg_G(v)=2$, $\deg _G (w) \geq 3$ and  $vw \in E(G)$.\\
By arguing as before (using induction and Facts~\ref{fact1} and~\ref{fact2}) we have that
\begin{align*}
{\rm MIS} (G) & \leq {\rm MIS}_G (v)+\sum _{u \in N_G(v)} {\rm MIS}_G (u) \leq 2 \times 2^{(n-3)/2-(k-2D)/(100D^2)} + 2^{(n-4)/2-(k-3D)/(100D^2)}
\\ & < 2^{n/2-k/(100D^2)},
\end{align*} as desired. (Here we have used that $2 \cdot 2^{-3/2+1/50} +2^{-2+3/100} <1$.) 

\smallskip

Cases~1--4 imply that we may now assume that $G$ consists precisely of $2$-regular components and components of minimum degree at least $3$.

\noindent 
\textbf{Case 5:} There exist  $v,w\in V(G)$ such that $\deg_G(v)=3$, $\deg _G (w) \geq 4$ and  $vw \in E(G)$.\\
By arguing similarly to  before (using induction and Facts~\ref{fact1} and~\ref{fact2}) we have that
\begin{align*}
{\rm MIS} (G) & \leq {\rm MIS}_G (v)+\sum _{u \in N_G(v)} {\rm MIS}_G (u) \leq 3 \times 2^{(n-4)/2-(k-3D)/(100D^2)} + 2^{(n-5)/2-(k-4D)/(100D^2)}
\\ & < 2^{n/2-k/(100D^2)},
\end{align*}
 as desired. (Here we have used that $3 \cdot 2^{-2+3/100} +2^{-5/2+1/25} <1$.) 

\smallskip

We may now assume that $G$ consists only of $2$- and $3$-regular components and components of minimum degree at least $4$.
However, if there is a component of minimum degree at least $4$ then by arguing precisely as in  Case~3, we obtain that ${\rm MIS}(G) \leq 2^{n/2-k/(100D^2)}$. So we may now assume $G$ consists of $2$- and $3$-regular components only.

\smallskip

\noindent 
\textbf{Case 6:} $G$ contains a $3$-regular component.\\
Here we  use the fact that ${\rm MIS}(G) \leq {\rm MIS}(G\setminus \{v\})+{\rm MIS}(G\setminus (N_G (v) \cup \{v\}))$ for any $v \in V(G)$.
Indeed, by induction we have 
$${\rm MIS}(G) \leq 2^{(n-1)/2-(k-5/2)/(100D^2)}+2^{(n-4)/2-(k-7)/(100D^2)} < 2^{n/2-k/(100D^2)},$$
as desired. (Here we have used that $2^{-1/2+1/40}+2^{-2+7/100}<1$.)

\smallskip

\noindent 
\textbf{Case 7:} $G$ is $2$-regular.\\
Since $G$ is triangle-free, Corollary~\ref{cycle} implies that ${\rm MIS}(G) \leq 2^{0.49n} \leq 2^{n/2-k/(100D^2)}$, as desired.
\qed

Finally, we show that Corollary~\ref{use} follows from Lemma~\ref{dense}.

\noindent
{\bf Proof of Corollary~\ref{use}.}
Every maximal independent set in $G$ can be obtained in the following two steps: 

(1) Choose an independent set $S\subseteq T$.

(2) Extend $S$ in $V(G)\setminus T=V(G')$, i.e.~choose a set $R\subseteq V(G')$ such that $R\cup S$ is a maximal independent set in $G$.

Note that although every maximal independent set in $G$ can be obtained in this way, it is not necessarily the case that given an arbitrary independent set $S \subseteq T$, there exists a set $R \subseteq V(G')$ such that $R \cup S$ is  a maximal independent set in $G$.
Notice that if $R\cup S$ is maximal,  $R$ is also a maximal independent set in $G'':=G\setminus (T\cup N_G(S))$. The number of choices for $S$ in (1) is at most $2^{|T|}$. 
Note that $G''$ is triangle-free, $\Delta (G'') \leq D$ and $e(G'') \geq e(G') - |T|D^2 \geq |G''|/2+ (k-|T|D^2)$. Thus, Lemma~\ref{dense} implies that 
 the number of extensions in (2) is at most $2^{n/2-(k-|T|D^2)/(100D^2)}$. Therefore, we have ${\rm{MIS}}(G)\le 2^{|T|}\cdot 2^{n/2-(k-|T|D^2)/(100D^2)}$, as desired.
\qed

 \end{appendix}

\end{document}